\documentclass[preprint,10pt]{elsarticle}

\usepackage[a4paper,left=18mm,right=18mm]{geometry}

\biboptions{numbers,sort&compress}

\usepackage{amsmath,amssymb,amsthm,booktabs,mathtools}
\usepackage{adjustbox}
\usepackage{microtype}
\usepackage{needspace}
\usepackage{hyperref}
\hypersetup{colorlinks=true,linkcolor=blue,citecolor=blue,urlcolor=blue,pdfauthor={Jeong-Hoon Ju, Hayoung Choi, Taehyeong Kim}}

\newtheorem{theorem}{Theorem}[section]
\newtheorem{corollary}[theorem]{Corollary}

\newtheorem{definition}[theorem]{Definition}

\newtheorem{remark}[theorem]{Remark}

\numberwithin{equation}{section}

\begin{document}
\begin{frontmatter}
\def\floatpagepagefraction{1}
\def\textpagefraction{.001}

\title{Improved Upper Bounds for the Rank of the Determinant Tensor via Glynn's Formula}

\author[1]{Jeong-Hoon Ju}
\ead{jeonghoon.ju@math.ku.dk}

\author[2,3]{Hayoung Choi}
\ead{hayoung.choi@knu.ac.kr}

\author[3]{Taehyeong Kim\corref{cor1}}

\ead{thkim0519@knu.ac.kr}

\affiliation[1]{organization={Department of Mathematical Sciences, University of Copenhagen},
            addressline={5 Universitetsparken},
            city={Copenhagen {\O}},
            postcode={2100},
            country={Denmark}}

\affiliation[2]{organization={Department of Mathematics, Kyungpook National University},
            city={Daegu},
            postcode={41566},
            country={Republic of Korea}}

\affiliation[3]{organization={Nonlinear Dynamics and Mathematical Application Center, Kyungpook National University},
            city={Daegu},
            postcode={41566},
            country={Republic of Korea}}

\cortext[cor1]{Corresponding author.}

\begin{abstract}
    In this paper, we improve the upper bound for the tensor rank of the $n \times n$ determinant tensor to $\mathcal{B}_n-2^{n-1}+n+1$ by finding a new tensor decomposition of the tensor, where $\mathcal{B}_n$ denotes the $n$-th Bell number. The main idea is to replace a part of the tensor decomposition suggested by Houston--Goucher--Johnston with a tensor that is essentially the same as the $n \times n$ permanent tensor, and to use Glynn's formula. This upper bound gives an improved upper bound for the Waring rank of the $n \times n$ determinant polynomial for $n \geq 17$ as well.
\end{abstract}


\end{frontmatter}

\section{Introduction}

A tensor is a multilinear function from a product of vector spaces to the base field. That is, a tensor is a generalization of bilinear function. For a bilinear function, its rank is defined to be the rank of the corresponding matrix. As a generalization of the notion of matrix rank, the \emph{tensor rank} of a tensor $T$ is defined by the smallest number of decomposable tensors required to express $T$ as their sum, and denoted by $\mathbf{R}(T)$ (see Definition \ref{definition:tensor rank}). In addition, as a generalization of the notion of rank of a symmetric matrix, the \emph{Waring rank} of a homogeneous polynomial $F$ is defined by the smallest number of linear forms required to express $F$ as a linear combination of them, and denoted by $\mathbf{R}_S(F)$ (see Definition \ref{definition:Waring rank}). Tensor rank and Waring rank are of interest in algebraic complexity theory, as they can be viewed as arithmetic circuit complexity measures associated with set-multilinear $\Sigma\Pi\Sigma$ circuits and $\Sigma\Lambda\Sigma$ circuits, respectively (see \citep{shpilka2010arithmetic, landsberg2017geometry}). They are also of interest in algebraic geometry, as their close connections with secant varieties of the Segre and Veronese varieties, respectively (see \citep{landsberg2012tensors}). Although there are several efficient ways to compute the rank of a matrix, they have not sufficiently generalized for tensor rank and Waring rank, so much effort has been devoted to improving the upper and lower bounds for these quantities.

In this paper, we improve the upper bounds for the tensor rank of the $n \times n$ \emph{determinant tensor}
\begin{equation*}
\operatorname{det}_n:=\sum_{\sigma \in \mathfrak{S}_n} \operatorname{sgn}(\sigma)~e_{\sigma(1)}\otimes {\cdots} \otimes e_{\sigma(n)} \in (\mathbb{K}^n)^{\otimes n},
\end{equation*}
where $\{e_1,...,e_n\}$ is the standard basis of $\mathbb{K}^n$, and the Waring rank of the $n \times n$ \emph{determinant polynomial}
\begin{equation*}
\operatorname{Det}_n:=\sum_{\sigma \in \mathfrak{S}_n} \operatorname{sgn}(\sigma)\prod_{i=1}^nx_{i,\sigma(i)} \in \mathbb{K}[x_{ij}~|~1 \leq i,j \leq n]_n,
\end{equation*}
where $x_{ij}$'s are independent variables.
Here, $\mathfrak{S}_n$ denotes the symmetric group of $n$ letters. 
Since the determinant is linear with respect to each column (or each row) of the matrix, then it is a tensor. If we write this tensor with respect to the standard basis of $\mathbb{K}^n$, then we obtain the expression of $\operatorname{det}_n$ above. Moreover, if we evaluate a generic $n \times n$ matrix $X=[x_{ij}]_{1 \leq i,j \leq n}$ to the determinant tensor $\operatorname{det}_n$, then we obtain the determinant polynomial $\operatorname{Det}_n$. Despite the fundamental nature of these objects, the tensor rank of $\operatorname{det}_n$ and the Waring rank of $\operatorname{Det}_n$ remain largely undetermined.

The main result of this paper is presented in the following theorem. 

\begin{theorem}\label{thm:Main Theorem}
    Let $\mathcal{B}_n$ denote the $n$-th Bell number. If $\operatorname{char}(\mathbb{K})=0$, then the following upper bounds hold:
    \begin{itemize}
        \item [(\romannumeral1)] $\mathbf{R}(\operatorname{det}_n) \leq \mathcal{B}_n-2^{n-1}+n+1$.
        \item [(\romannumeral2)] $\mathbf{R}_S(\operatorname{Det}_n) \leq 2^{n-1} \cdot (\mathcal{B}_n-2^{n-1}+n+1)$.
    \end{itemize}
\end{theorem}


We prove that a part of the tensor decomposition suggested in \citep{houston2024bell} is essentially the same as the $n \times n$ permanent tensor (see (\ref{eq:permanent tensor})), and modify it using Glynn's formula (see Theorem \ref{theorem:glynn's formula}). This gives a new tensor decomposition of $\operatorname{det}_n$ described in Theorem \ref{thm:key Theorem}, and the upper bound for $\mathbf{R}(\operatorname{det}_n)$. We get the upper bound for $\mathbf{R}_S(\operatorname{Det}_n)$ using the relation (\ref{eq:inequality between Waring rank and tensor rank}) below. The comparisons of upper bounds above with other previous upper bounds are presented in Tables \ref{table1} and \ref{table2} for small $n$. Theorem \ref{thm:Main Theorem} gives the best known upper bound of the tensor rank for all $n \geq 3$, and the best upper bound of the Waring rank for $n \geq 17$.

\begin{table}[h!]
\centering
\begin{adjustbox}{max width=\linewidth}\begin{tabular}{ |c|c|c|c|c|c|c|c|c| }
\hline
 $n$  & 3 & 4 & 5 & 6 & 7 & 8 & ${\cdots}$ \\
 \hline
 Houston--Goucher--Johnston \citep{houston2024bell}  & \textbf{5} & 15 & 52 & 203 & 877 & 4140 & ${\cdots}$\\
 \hline
 Ju--Kim--Kim \citep{ju2025finding, ju2025symmetries}  & 6 & \textbf{12} & 60 & 180 & 1260 & 5040 & ${\cdots}$\\
 \hline
 Theorem \ref{thm:Main Theorem} & \textbf{5} & \textbf{12} & \textbf{42} & \textbf{178} & \textbf{821} & \textbf{4021} & ${\cdots}$\\
 \hline
\end{tabular}\end{adjustbox}
\caption{Comparison of upper bounds for $\mathbf{R}(\operatorname{det}_n)$. For each $n$, the best bound is in bold.}
\label{table1}
\end{table}
\begin{table}[h!]
\centering
\begin{adjustbox}{max width=\linewidth}\begin{tabular}{ |c|c|c|c|c|c|c|c|c| }
\hline
 $n$  & 3 & 4 & 5 & ${\cdots}$ & 16 & 17 & ${\cdots}$ \\
 \hline
  Johns--Teitler \citep{johns2022}  & \textbf{18} & \textbf{96} & \textbf{600} & ${\cdots}$ & $\boldsymbol{\approx 3.35 \times 10^{14}}$  & $\approx 6.05 \times 10^{15}$ & ${\cdots}$\\
 \hline
 Houston--Goucher--Johnston \citep{houston2024bell}  & 20 & 120 & 832 & ${\cdots}$ & $\approx 3.43 \times 10^{14}$ & $\approx 5.43 \times 10^{15}$ & ${\cdots}$\\
 \hline
 Theorem \ref{thm:Main Theorem} & 20 & \textbf{96} & 672 & ${\cdots}$ & $\approx 3.43 \times 10^{14}$ & $\boldsymbol{\approx 5.43 \times 10^{15}}$ & ${\cdots}$\\
 \hline
\end{tabular}\end{adjustbox}
\caption{Comparison of upper bounds for $\mathbf{R}_S(\operatorname{Det}_n)$. The bounds by Johns--Teitler \citep{johns2022} hold over $\mathbb{C}$. For each $n$, the best bound is in bold.}
\label{table2}
\end{table}

The organization of this paper is as follows. In Section \ref{section:Preliminaries}, we review the notions of tensor rank and Waring rank, then some previous results relevant to Theorem \ref{thm:Main Theorem}. In Section \ref{section:Determinant formula using Glynn's formula}, we prove Theorem \ref{thm:Main Theorem} over an arbitrary field with characteristic zero. In Section \ref{section:Fields of nonzero characteristic}, we deal with the ranks over  fields of nonzero characteristic. Finally, in Section \ref{section:Further directions}, we suggest several directions for further studies.

\section{Preliminaries}\label{section:Preliminaries}

Throughout this paper, we assume that the base field $\mathbb{K}$ is of characteristic zero unless otherwise stated, and we use the following notations: $\mathfrak{S}_n$ denotes the symmetric group of $n$ letters, $[n]:=\{1,2,...,n\}$, $\mathbf{1}_{n \times n}$ denotes the $n \times n$ matrix with all entries being $1$, and $I_{n \times n}$ denotes the $n \times n$ identity matrix.

\subsection{Tensor rank and Waring rank}

\begin{definition}[Tensor rank]\label{definition:tensor rank}
    Let $V_1,...,V_d$ be finite-dimensional $\mathbb{K}$-vector spaces, and let $T \in V_1 \otimes {\cdots} \otimes V_d$ be nonzero. The number
    \begin{equation*}
\adjustbox{max width=\linewidth}{$\displaystyle
\mathbf{R}(T):=\min\left\{k~\middle|~T=\sum_{i=1}^k v_{i,1} \otimes {\cdots} \otimes v_{i,d},~\text{where}~v_{i,j} \in V_j~\text{for each}~i \in [k]~\text{and}~j \in [d]\right\}
$}
\end{equation*}
    is called the \emph{tensor rank} of $T$. We refer to such a formula of $T$ on the right-hand side as a \emph{tensor decomposition} of $T$.
\end{definition}

\begin{definition}[Waring rank]\label{definition:Waring rank}
    Let $F \in \mathbb{K}[x_1,...,x_n]_d$, that is, $F$ is a homogeneous polynomial of degree $d$. The number
    \begin{equation*}
\adjustbox{max width=\linewidth}{$\displaystyle
\mathbf{R}_S(F):=\min\left\{k~\middle|~F=\sum_{i=1}^kc_i(l_i)^d,~\text{where}~c_i \in \mathbb{K}~\text{and}~l_i\in \mathbb{K}[x_1,...,x_n]_1~\text{for each}~i \in [k]\right\}
$}
\end{equation*}
    is called the \emph{Waring rank} of $F$.
\end{definition}

The Waring rank is often called the symmetric rank, because a homogeneous polynomial corresponds to a symmetric tensor (which is invariant under an arbitrary transposition of variables) and the Waring rank is the same as the smallest number of symmetric rank one tensors required to express the symmetric tensor as a linear combination of these tensors. We refer to \citep{landsberg2012tensors} for more details. 

For the determinant tensor $\operatorname{det}_n$ and the determinant polynomial $\operatorname{Det}_n$, we obtain the following inequality
\begin{equation}\label{eq:inequality between Waring rank and tensor rank}
\mathbf{R}_S(\operatorname{Det}_n) \leq 2^{n-1}\cdot \mathbf{R}(\operatorname{det}_n)
\end{equation}
from the fact $\mathbf{R}_S(x_1{\cdots} x_n)=2^{n-1}$ (see \citep{carlini2012solution}). We refer to \citep{houston2024bell} for more detailed argument. 

\subsection{Previous relevant results}

For $n \geq 7$, the previous best upper bound of $\mathbf{R}(\operatorname{det}_n)$ was the $n$-th Bell number $\mathcal{B}_n$ by \citep{houston2024bell}. Since we will use this result, we explain it in more detail. A \emph{partial partition} of $[n]$ is a set of disjoint subsets of $[n]$, and a subset in a partial partition is called a \emph{part}. We let $\operatorname{PP}(n)$ denote the set of partial partitions of $[n]$ without singletons. Since there is a natural 1-1 correspondence between partial partitions of $[n]$ with no singletons and partitions of $[n]$, we obtain $|\operatorname{PP}(n)|=\mathcal{B}_n$. For $i,j \in [n]$, we let $i \underset{P}{\sim} j$ mean that there is a part in $P$ containing both $i$ and $j$. Then, for $k \in [n]$, it is guaranteed that $k \underset{P}{\sim} k$ if and only if there is a part containing $k$. The main result in \citep{houston2024bell} is as in the following theorem.

\begin{theorem}[see \citep{houston2024bell}]\label{theorem:bell number}
    The $n \times n$ determinant tensor has the following tensor decomposition:
    \begin{equation}\label{eq:Houston formula}
\operatorname{det}_n=\sum_{P \in \operatorname{PP}(n)} \operatorname{sgn}(P)\cdot|P|!\bigotimes_{i=1}^nv_i(P),
\end{equation}
    where $\operatorname{sgn}(P)=\prod_{S \in P}(-1)^{|S|+1}$ and 
\begin{equation*}
v_i(P)=\begin{cases}
            \sum\limits_{j \underset{P}{\sim} i, j \neq i} e_j &~\text{if}~i \underset{P}{\sim} i\\
            e_i+\sum\limits_{j \underset{P}{\sim} j} e_j & \text{if}~i \underset{P}{\nsim} i
        \end{cases}.
\end{equation*}
    In particular, $\mathbf{R}(\operatorname{det}_n) \leq \mathcal{B}_n$.
\end{theorem}

For example, when $n=4$, we have the following tensor decomposition from (\ref{eq:Houston formula}):
\begin{equation}\label{eq:Houston formula of det4}
\begin{aligned}
        \operatorname{det}_4=&e_1 \otimes e_2 \otimes e_3 \otimes e_4\\
        &-e_2 \otimes e_1 \otimes (e_1+e_2+e_3) \otimes (e_1+e_2+e_4)\\
        &-e_3 \otimes (e_1+e_2+e_3) \otimes e_1 \otimes (e_1+e_3+e_4)\\
        &-e_4 \otimes (e_1+e_2+e_4) \otimes (e_1+e_3+e_4) \otimes e_1\\
        &-(e_1+e_3+e_4) \otimes (e_2+e_3+e_4) \otimes e_4 \otimes e_3\\
        &-(e_1+e_2+e_4) \otimes e_4 \otimes (e_2+e_3+e_4) \otimes e_2\\
        &-(e_1+e_2+e_3) \otimes e_3 \otimes e_2 \otimes (e_2+e_3+e_4)\\
        &+(e_2+e_3) \otimes (e_1+e_3) \otimes (e_1+e_2) \otimes (e_1+e_2+e_3+e_4)\\
        &+(e_2+e_4) \otimes (e_1+e_4) \otimes (e_1+e_2+e_3+e_4) \otimes (e_1+e_2)\\
        &+(e_3+e_4) \otimes (e_1+e_2+e_3+e_4) \otimes (e_1+e_4) \otimes (e_1+e_3)\\
        &+(e_1+e_2+e_3+e_4) \otimes (e_3+e_4) \otimes (e_2+e_4) \otimes (e_2+e_3)\\
        &-(e_2+e_3+e_4) \otimes (e_1+e_3+e_4) \otimes (e_1+e_2+e_4) \otimes (e_1+e_2+e_3)\\
        &+2e_2 \otimes e_1 \otimes e_4 \otimes e_3\\
        &+2e_3 \otimes e_4 \otimes e_1 \otimes e_2\\
        &+2e_4 \otimes e_3 \otimes e_2 \otimes e_1.
    \end{aligned}
\end{equation}

In the case of $4 \times 4$ and $6 \times 6$ determinant tensors, the upper bounds for their tensor ranks improved upon the Bell-number bounds in \citep{ju2025finding, ju2025symmetries}. However, the bounds in \citep{ju2025symmetries} for $n \geq 7$ are bigger than the Bell numbers $\mathcal{B}_n$ as in Table \ref{table1}. On the other hand, if $n \geq 17$, then the previous best upper bound of $\mathbf{R}_S(\operatorname{Det}_n)$ was $2^{n-1}\cdot \mathcal{B}_n$ by Theorem \ref{theorem:bell number} and the relation (\ref{eq:inequality between Waring rank and tensor rank}). If $n \leq 16$ and $\mathbb{K}=\mathbb{C}$, then the best upper bound of $\mathbf{R}_S(\operatorname{Det}_n)$ was $n \cdot n!$ by \citep{johns2022} as in Table \ref{table2}.

For the proof of Theorem \ref{thm:key Theorem} below, we introduce the $n \times n$ permanent tensor $\operatorname{perm}_n$. Since the permanent is linear with respect to each column (or row) of the matrix, then it is a tensor. If we write it with respect to the standard basis $\{e_1,...,e_n\}$ of $\mathbb{K}^n$, then we obtain the expression
\begin{equation}\label{eq:permanent tensor}
\operatorname{perm}_n=\sum_{\sigma \in \mathfrak{S}_n} e_{\sigma(1)}\otimes {\cdots} \otimes e_{\sigma(n)} \in (\mathbb{K}^n)^{\otimes n}.
\end{equation}
If we evaluate a generic $n \times n$ matrix $X=[x_{ij}]_{1 \leq i,j \leq n}$ to $\operatorname{perm}_n$, then we obtain the $n \times n$ permanent polynomial $\operatorname{Perm}_n \in \mathbb{K}[x_{ij}~|~1 \leq i,j \leq n]_n$. Glynn proved that $\operatorname{Perm}_n$ admits an explicit formula.

\begin{theorem}[Glynn's formula, see \citep{glynn2010permanent}]\label{theorem:glynn's formula}
    Let $\mathcal{I}_n$ be a set defined by
    \begin{equation}\label{eq:index set I}
\mathcal{I}_n=\{\delta=(\delta_1,\delta_2,\ldots,\delta_n) \in \{\pm 1\}^{n}~|~\delta_1=1\}.
\end{equation}
    Then the $n \times n$ permanent polynomial can be expressed by
    \begin{equation*}
\operatorname{Perm}_n=\frac{1}{2^{n-1}}\sum_{\delta \in \mathcal{I}_n}\left(\prod_{k=1}^n \delta_k\right) \prod_{j=1}^n \left(\sum_{i=1}^n\delta_ix_{ij}\right).
\end{equation*}
    We call it Glynn's formula.
\end{theorem}

\begin{remark}\label{remark:a variant of Glynn's formula}
    For an $n \times n$ matrix $H=[h_{ij}]_{1 \leq i,j \leq n}$, we consider the Hadamard product $H \circ X=[h_{ij}x_{ij}]_{1 \leq i,j \leq n}$. Since
    \begin{equation*}
\operatorname{perm}_n(H \circ X)=\frac{1}{2^{n-1}}\sum_{\delta \in \mathcal{I}_n}\left(\prod_{k=1}^n \delta_k\right) \prod_{j=1}^n \left(\sum_{i=1}^n\delta_ih_{ij}x_{ij}\right)
\end{equation*}
    by Glynn's formula, we obtain
    \begin{equation*}
\begin{aligned}\operatorname{perm}_{n,H}&:=\sum_{\sigma \in \mathfrak{S}_n} \left(\prod_{i=1}^n h_{i, \sigma(i)}\right)e_{\sigma(1)} \otimes e_{\sigma(2)} \otimes {\cdots} \otimes e_{\sigma(n)}\\&=\frac{1}{2^{n-1}}\sum_{\delta \in \mathcal{I}_n}\left(\prod_{k=1}^n \delta_k\right) \bigotimes_{j=1}^n \left(\sum_{i=1}^n\delta_ih_{ij}e_i\right).
\end{aligned}\end{equation*}
    In particular, $\mathbf{R}(\operatorname{perm}_{n,H})\leq 2^{n-1}$.
    We will use this property when we prove Theorem \ref{thm:key Theorem}. Since Glynn's formula holds over an arbitrary field of characteristic $\neq 2$, we will use this property in Section \ref{section:Fields of nonzero characteristic} as well.
\end{remark}

\section{Determinant formula using Glynn's formula}\label{section:Determinant formula using Glynn's formula}

Theorem \ref{thm:key Theorem} below is motivated by the comparison of $\operatorname{det}_4$ and $\operatorname{perm}_4$. 
Over a field of characteristic $2$, $|P|!=0$ whenever $|P|\ge 2$, so every such term in the decomposition \eqref{eq:Houston formula of det4} of $\operatorname{det}_4$ vanishes. The remaining terms are the term for the empty partial partition and the terms for partial partitions consisting of a single part. In this characteristic, $\operatorname{det}_4=\operatorname{perm}_4$.
Hence, if we collect the terms corresponding to partial partitions consisting of a single part or an empty set, then we obtain the tensor decomposition of $\operatorname{perm}_4$ over a field of characteristic $2$. We prove that such a property holds for all $n$ and even over a field of characteristic zero.

\begin{theorem}\label{thm:key Theorem}
    Let $H=[h_{ij}]_{1 \leq i,j \leq n}=\mathbf{1}_{n \times n}-2I_{n \times n}$, and let $\mathcal{I}_n$ be a set defined in (\ref{eq:index set I}). Then the $n \times n$ determinant tensor has the following tensor decomposition:
    \begin{equation}\label{eq:new tensor decomposition using Glynn's formula}
\begin{aligned}
\operatorname{det}_n=2e_1 \otimes e_2 \otimes {\cdots} \otimes e_n  +\left(\sum_{\substack{P \in \operatorname{PP}(n) \\ |P| \geq 2}} \operatorname{sgn}(P)\cdot |P|!\bigotimes_{i=1}^n v_i(P)\right)
        +\frac{(-1)^{n+1}}{2^{n-1}}\sum_{\delta \in \mathcal{I}_n}\left(\prod_{k=1}^n \delta_k\right) \bigotimes_{j=1}^n \left(\sum_{i=1}^n\delta_ih_{ij}e_i\right),
\end{aligned}
\end{equation}
    where $\operatorname{sgn}(P)$ and $v_i(P)$ are defined as in Theorem \ref{theorem:bell number}. 
\end{theorem}
\begin{proof}
     If we prove that the $n \times n$ determinant tensor has the following tensor decomposition:
    \begin{equation}\label{eq:new tensor decomposition}
\begin{aligned}
\operatorname{det}_n=2e_1 \otimes e_2 \otimes {\cdots} \otimes e_n +\left(\sum_{\substack{P \in \operatorname{PP}(n) \\ |P| \geq 2}} \operatorname{sgn}(P)\cdot |P|!\bigotimes_{i=1}^n v_i(P)\right)  + (-1)^{n+1}\operatorname{perm}_{n,H},
\end{aligned}
\end{equation}
    where 
    \begin{equation*}
\begin{aligned}
\operatorname{perm}_{n,H}&:=\sum_{\sigma \in \mathfrak{S}_n} \left(\prod_{i=1}^n h_{i, \sigma(i)}\right)e_{\sigma(1)} \otimes e_{\sigma(2)} \otimes {\cdots} \otimes e_{\sigma(n)},
\end{aligned}
\end{equation*}
    then the assertion holds by Glynn's formula (see Remark \ref{remark:a variant of Glynn's formula}). From now on, we prove (\ref{eq:new tensor decomposition}).
    In the tensor decomposition (\ref{eq:Houston formula}), we consider the terms for $P \in \operatorname{PP}(n)$ consisting of a single part:
    \begin{equation}\label{eq:summands corresponding to single part partial partitions}
\sum_{\{S\} \in \operatorname{PP}(n)} \operatorname{sgn}(\{S\})\bigotimes_{i=1}^n v_i(\{S\}).
\end{equation}
    Here, we have $\operatorname{sgn}(\{S\})=(-1)^{|S|+1}$ and 
\begin{equation*}
v_i(\{S\})=\begin{cases}
        \sum_{j \in S \setminus \{i\}}e_j&\text{if}~i \in S\\
        e_i+\sum_{j \in S}e_j&\text{if}~i \notin S
    \end{cases}.
\end{equation*} For each function $f:[n] \rightarrow [n]$, we associate a rank-one tensor
    \begin{equation}\label{eq:decomposable tensor corresponding to f}
e_{f(1)}\otimes e_{f(2)} \otimes {\cdots} \otimes e_{f(n)}.
\end{equation}
    We claim that a nonzero constant multiple of the rank-one tensor (\ref{eq:decomposable tensor corresponding to f}) appears in the expansion of $\bigotimes_{i=1}^n v_i(\{S\})$ if and only if the function $f$ satisfies the inclusions:
    \begin{equation}\label{eq:key inclusions}
\{f(i)~|~f(i) \neq i\} \subset S \subset [n] \setminus \{i~|~f(i)=i\}.
\end{equation}
    
    Assume that the hypothesis of the claim holds. Then, for each $i \in [n]$, the vector $e_{f(i)}$ is a summand of $v_i(\{S\})$. If $k \in S$, then $e_{f(k)} \in \{e_j~|~j \in S \setminus \{k\}\}$, and so $f(k) \neq k$. Hence, we obtain $S \subset [n] \setminus \{i~|~f(i)=i\}$. Now, we show that $\{f(i)~|~f(i) \neq i\} \subset S$. Let $k \in [n]$ satisfy $f(k) \neq k$. There are two cases: either $k \in S$ or $k \notin S$. If $k \in S$, then $e_{f(k)} \in \{e_j~|~j \in S \setminus \{k\}\}$, and so $f(k) \in S \setminus \{k\}$. If $k \notin S$, then $e_{f(k)} \in \{e_k\} \cup \{e_j~|~j \in S\}$, and so $f(k) \in \{k\} \cup S$. Since $f(k) \neq k$, we obtain $f(k) \in S$ in this case. 
    
    Now, we prove the converse. Assume that $S$ satisfies (\ref{eq:key inclusions}). Consider $v_k(\{S\})$ for a $k \in [n]$. First, assume that $k \in S$. Since $S \subset [n] \setminus \{i~|~f(i)=i\}$, we obtain $f(k) \neq k$. Since $\{f(i)~|~f(i) \neq i\} \subset S$, then $f(k) \in S$. Hence, $e_{f(k)}$ is a summand of $v_k(\{S\})=\sum_{j \in S \setminus \{k\}} e_j$. Now, assume that $k \notin S$. There are two cases: either $f(k)=k$ or $f(k) \neq k$. If $f(k)=k$, then $e_{f(k)}$ is a summand of $v_k(\{S\})=e_k+\sum_{j \in S}e_j=e_{f(k)}+\sum_{j \in S}e_j$. If $f(k) \neq k$, then $f(k) \in S$, and so $e_{f(k)}$ is a summand of $v_k(\{S\})=e_{k}+\sum_{j \in S}e_j$. Hence, the claim holds.

    If $c_fe_{f(1)} \otimes {\cdots} \otimes e_{f(n)}$ (where $c_f$ is a constant) appears in the expansion of (\ref{eq:summands corresponding to single part partial partitions}), then 
    \begin{equation*}
c_f=\sum_{\substack{S \subset [n]\\ (\ref{eq:key inclusions})~\text{holds}}} (-1)^{|S|+1}.
\end{equation*}
    If $\{f(i)~|~f(i) \neq i\} \cup \{i~|~f(i)=i\} \subsetneq [n]$, then there exists an element $s \in [n] \setminus (\{f(i)~|~f(i) \neq i\} \cup \{i~|~f(i)=i\})$ so that
    \begin{equation*}
c_f=\sum_{\substack{s \in S \subset [n]\\ (\ref{eq:key inclusions})~\text{holds}}} (-1)^{|S|+1}+\sum_{\substack{s \notin S \subset [n]\\ (\ref{eq:key inclusions})~\text{holds}}} (-1)^{|S|+1}=0.
\end{equation*}
    If $\{f(i)~|~f(i) \neq i\} \cup \{i~|~f(i)=i\} = [n]$, then $f$ is a permutation, and so $S \subset [n]$ satisfying (\ref{eq:key inclusions}) is uniquely determined as $S=\{f(i)~|~f(i) \neq i\}$. Hence, $c_f=(-1)^{|S|+1}=(-1)^{n-|\{i~|~f(i)=i\}|+1} \neq 0$ in this case. Thus, the expansion of (\ref{eq:summands corresponding to single part partial partitions}) is
    \begin{equation}\label{eq:the expansion}
\begin{aligned}&\left(\sum_{\sigma \in \mathfrak{S}_n} (-1)^{n-|\{i~|~\sigma(i)=i\}|+1}e_{\sigma(1)} \otimes {\cdots} \otimes e_{\sigma(n)}\right)\\&\quad -(-1)^{n-n+1}e_{1} \otimes e_{2}\otimes {\cdots} \otimes e_{n},
\end{aligned}
\end{equation}
    because the last term $(-1)^{n-n+1}e_{1} \otimes e_{2}\otimes {\cdots} \otimes e_{n}$ corresponds to the empty partial partition. On the other hand, the entries $h_{ij}$ of $H$ satisfy
    \begin{equation*}
\prod_{i=1}^n h_{i,\sigma(i)}=(-1)^{|\{i~|~\sigma(i)=i\}|}.
\end{equation*}
    Therefore, the formula (\ref{eq:the expansion}) is the same as
    \begin{equation*}
(-1)^{n+1}\operatorname{perm}_{n,H}+~e_1 \otimes e_2 \otimes {\cdots} \otimes e_n,
\end{equation*}
    and so the equation (\ref{eq:new tensor decomposition}) follows.
\end{proof}

We obtain Theorem \ref{thm:Main Theorem} as a corollary of Theorem \ref{thm:key Theorem}, if we use the relation (\ref{eq:inequality between Waring rank and tensor rank}) additionally for the Waring rank.

\begin{remark}\label{remark:comparison with another tensor decomposition}
    Let $\mathbb{K}=\mathbb{R}$ or $\mathbb{C}$. The tensor decomposition
    \begin{equation*}\label{eq:Derksen's formula}
\begin{aligned}
			\operatorname{det}_3=&\frac{1}{2}((e_2+e_3) \otimes (e_1-e_2) \otimes (e_1+e_2)\\
			&+(e_1+e_2) \otimes (e_2-e_3) \otimes (e_2+e_3)\\
			&-2e_2\otimes (e_1-e_3) \otimes (e_1+e_3)\\
			&+(e_2-e_3) \otimes (e_1+e_2) \otimes (e_1-e_2)\\
			&-(e_1-e_2) \otimes (e_2+e_3) \otimes (e_2-e_3))
		\end{aligned}
\end{equation*}
    of $\operatorname{det}_3$ suggested in \citep{derksen2016nuclear} can be obtained by the diagonal action of $A_3:=\frac{1}{\sqrt[3]{4}}\begin{bmatrix}
            0 & 1 & 1\\
            -2 & 0 & 0\\
            0 & -1 & 1
        \end{bmatrix} \in \operatorname{SL}_3$
    on the tensor decomposition (\ref{eq:new tensor decomposition using Glynn's formula}) for $n=3$:
    \begin{equation*}
\begin{aligned}
            \operatorname{det}_3&=2e_1\otimes e_2 \otimes e_3\\
            &+\frac{1}{4}(-e_1+e_2+e_3) \otimes (e_1-e_2+e_3) \otimes (e_1+e_2-e_3)\\
            &-\frac{1}{4}(-e_1+e_2-e_3) \otimes (e_1-e_2-e_3) \otimes (e_1+e_2+e_3)\\
            &-\frac{1}{4}(-e_1-e_2+e_3) \otimes (e_1+e_2+e_3) \otimes (e_1-e_2-e_3)\\
            &+\frac{1}{4}(-e_1-e_2-e_3) \otimes (e_1+e_2-e_3) \otimes (e_1-e_2+e_3).
        \end{aligned}
\end{equation*}
    Furthermore, the tensor decomposition
    \begin{equation*}
\begin{aligned}
		{\operatorname{det}}_4= &\frac{1}{2}((e_1-e_2)\otimes (e_3-e_4)\otimes (e_3+e_4)\otimes (e_1+e_2)\\
			&-(e_1-e_3)\otimes (e_2-e_4)\otimes (e_2+e_4)\otimes (e_1+e_3)\\
			&+(e_1-e_4)\otimes (e_2-e_3)\otimes (e_2+e_3)\otimes (e_1+e_4)\\
			&+(e_2-e_3)\otimes (e_1-e_4)\otimes (e_1+e_4)\otimes (e_2+e_3)\\
			&-(e_2-e_4)\otimes (e_1-e_3)\otimes (e_1+e_3)\otimes (e_2+e_4)\\
			&+(e_3-e_4)\otimes (e_1-e_2)\otimes (e_1+e_2)\otimes (e_3+e_4)\\
			&+(e_1+e_2)\otimes (e_3+e_4)\otimes (e_3-e_4)\otimes (e_1-e_2)\\
			&-(e_1+e_3)\otimes (e_2+e_4)\otimes (e_2-e_4)\otimes (e_1-e_3)\\
			&+(e_1+e_4)\otimes (e_2+e_3)\otimes (e_2-e_3)\otimes (e_1-e_4)\\
			&+(e_2+e_3)\otimes (e_1+e_4)\otimes (e_1-e_4)\otimes (e_2-e_3)\\
			&-(e_2+e_4)\otimes (e_1+e_3)\otimes (e_1-e_3)\otimes (e_2-e_4)\\
			&+(e_3+e_4)\otimes (e_1+e_2)\otimes (e_1-e_2)\otimes (e_3-e_4))
	\end{aligned}
\end{equation*}
    of $\operatorname{det}_4$ suggested in \citep{ju2025finding} can be obtained by the diagonal action of $A_4:=\frac{1}{\sqrt{2}}\begin{bmatrix}
            1 & 0 & 0 & 1\\
            0 & -1 & -1 & 0\\
            0 & 1 & -1 & 0\\
            -1 & 0 & 0 & 1
        \end{bmatrix} \in \operatorname{SL}_4$ 
    on the tensor decompositions (\ref{eq:new tensor decomposition using Glynn's formula}) for $n=4$:
    \begin{equation*}
\adjustbox{max width=\linewidth}{$\displaystyle
\begin{aligned}
            \operatorname{det}_4&=2e_1 \otimes e_2 \otimes e_3 \otimes e_4\\
            &+2e_2 \otimes e_1 \otimes e_4 \otimes e_3\\
            &+2e_3 \otimes e_4 \otimes e_1 \otimes e_2\\
            &+2e_4 \otimes e_3 \otimes e_2 \otimes e_1\\
            &-\frac{1}{8}(-e_1+e_2+e_3+e_4) \otimes (e_1-e_2+e_3+e_4) \otimes (e_1+e_2-e_3+e_4) \otimes (e_1+e_2+e_3-e_4)\\
            &+\frac{1}{8}(-e_1-e_2+e_3+e_4) \otimes (e_1+e_2+e_3+e_4) \otimes (e_1-e_2-e_3+e_4) \otimes (e_1-e_2+e_3-e_4)\\
            &+\frac{1}{8}(-e_1+e_2-e_3+e_4) \otimes (e_1-e_2-e_3+e_4) \otimes (e_1+e_2+e_3+e_4) \otimes (e_1+e_2-e_3-e_4)\\
            &+\frac{1}{8}(-e_1+e_2+e_3-e_4) \otimes (e_1-e_2+e_3-e_4) \otimes (e_1+e_2-e_3-e_4) \otimes (e_1+e_2+e_3+e_4)\\
            &-\frac{1}{8}(-e_1-e_2-e_3+e_4) \otimes (e_1+e_2-e_3+e_4) \otimes (e_1-e_2+e_3+e_4) \otimes (e_1-e_2-e_3-e_4)\\
            &-\frac{1}{8}(-e_1-e_2+e_3-e_4) \otimes (e_1+e_2+e_3-e_4) \otimes (e_1-e_2-e_3-e_4) \otimes (e_1-e_2+e_3+e_4)\\
            &-\frac{1}{8}(-e_1+e_2-e_3-e_4) \otimes (e_1-e_2-e_3-e_4) \otimes (e_1+e_2+e_3-e_4) \otimes (e_1+e_2-e_3+e_4)\\
            &+\frac{1}{8}(-e_1-e_2-e_3-e_4) \otimes (e_1+e_2-e_3-e_4) \otimes (e_1-e_2+e_3-e_4) \otimes (e_1-e_2-e_3+e_4).
        \end{aligned}
$}
\end{equation*}
\end{remark}

\Needspace{8\baselineskip}
\section{Fields of nonzero characteristic}\label{section:Fields of nonzero characteristic}

Let $\mathcal{B}_{n,k}$ denote the number of partial partitions of $[n]$ that contain exactly $k$ parts and no singletons. We refer to \citep{houston2024bell} and the references therein for the formulas of some $\mathcal{B}_{n,k}$. We obtain the following corollary from Theorem \ref{thm:key Theorem}.

\begin{corollary}
    Assume that $\operatorname{char}(\mathbb{K})=p>2$. Let $H=[h_{ij}]$ be the $n \times n$ matrix defined in Theorem \ref{thm:key Theorem}. Let $\mathcal{I}_n$ be a set defined in (\ref{eq:index set I}). Then the $n \times n$ determinant tensor has the following tensor decomposition:
    \begin{equation*}
\begin{aligned}
\operatorname{det}_n&=2e_1 \otimes e_2 \otimes {\cdots} \otimes e_n \\&\quad +\left(\sum_{\substack{P \in \operatorname{PP}(n) \\ 2 \leq |P| \leq p-1}} \operatorname{sgn}(P)\cdot |P|!\bigotimes_{i=1}^n v_i(P)\right)\\&\quad +\frac{(-1)^{n+1}}{2^{n-1}}\sum_{\delta \in \mathcal{I}_n}\left(\prod_{k=1}^n \delta_k\right) \bigotimes_{j=1}^n \left(\sum_{i=1}^n\delta_ih_{ij}e_i\right),
\end{aligned}
\end{equation*}
    where $\operatorname{sgn}(P)$ and $v_i(P)$ are defined as in Theorem \ref{theorem:bell number}. In particular,
    \begin{equation*}
\mathbf{R}(\operatorname{det}_n) \leq \sum_{k=0}^{p-1}\mathcal{B}_{n,k} -2^{n-1}+n+1.
\end{equation*}
\end{corollary}

\section{Further directions}\label{section:Further directions}

It has already been proved in \citep{derksen2019explicit, han2025recursive}, respectively, that $\mathbf{R}(\operatorname{det}_3)=5$ and $\mathbf{R}(\operatorname{det}_4)=12$. However, from Theorem \ref{thm:Main Theorem} and \citep{han2025recursive}, we only obtain the bounds $27 \leq \mathbf{R}(\operatorname{det}_5) \leq 42$. We leave the determination of its exact value as a problem for future work. 

The $n \times n$ determinant tensor $\operatorname{det}_n$ is invariant under the diagonal action of $\operatorname{SL}_n$. Hence, the action of $\operatorname{SL}_n$ on a tensor rank decomposition of $\operatorname{det}_n$ produces another decomposition of the same tensor. In Remark \ref{remark:comparison with another tensor decomposition}, we verified that the two known tensor rank decompositions of $\operatorname{det}_n$ are in the same orbit under the $\operatorname{SL}_n$-action for $n=3,4$. We leave the investigation of this orbit for general $n$ as a problem for future work.

\section*{Declaration of generative AI use}

During the preparation of this work, the authors used GPT 5.6 Sol solely to assist in obtaining the  matrices $A_3$ and $A_4$ in Remark \ref{remark:comparison with another tensor decomposition}. The authors take full responsibility for the content of the manuscript.

\section*{Acknowledgements}

The work of J.-H. Ju was supported by the Novo Nordisk Foundation (grant NNF20OC0059939). 
This work of H. Choi and T. Kim was supported by the National Research Foundation of Korea(NRF) grant funded by the Korea government(MSIT) (No. RS-2022-NR070839 and RS-2024-00342939). 
T. Kim was supported by the National Research Foundation of Korea(NRF) grant funded by the Korea government(MSIT) (No. RS-2025-25436769).


\end{document}